\documentclass[12pt]{amsart}

\usepackage{amsmath}
\usepackage{amssymb}
\usepackage{amsthm}
\usepackage{graphicx}
\usepackage{psfrag}
\usepackage{verbatim}

\usepackage[margin=1.2in]{geometry}
\def\E{\mathbb{E}}

\def\poiss{{\rm Poiss}}

\def\E{\mathbb{E}}
\def\R{\mathbb{R}}

\def\eps{\epsilon}
\def\del{\delta}
\def\gam{\gamma}

\def\1{\mathbf{1}}

\def\lam {\lambda}
\def\tce{t_c + \eps}
\def\tce2{t_c + \frac{\eps}{2}}
\def\ER{Erd\H{o}s-R\'{e}nyi }

\def\unsat{\mathrm{UNSAT}}

\newtheorem*{theorem*}{Theorem}
\newtheorem{theorem}{Theorem}

\newtheorem{prop}{Proposition}
\newtheorem*{prop*}{Proposition}
\newtheorem{conj}{Conjecture}
\newtheorem{claim}{Claim}

\def\al {\alpha}

\begin{document}
\title{On Sharp Thresholds in Random Geometric Graphs}
\author{Milan Bradonji\'c}\thanks{Bell Labs, Alcatel-Lucent, 600 Mountain Avenue 2C318, Murray Hill, New Jersey 07974, USA, E-mail: milan@research.bell-labs.com.} 
\author{ 
Will Perkins}\thanks{School of Mathematics, Georgia Tech, 686 Cherry St, Atlanta, GA 30332, E-mail: perkins@math.gatech.edu. Supported in part by an NSF Postdoctoral Fellowship.}

\maketitle
\begin{abstract}

%We study two geometric models of random $k$-satisfiability which combine random $k$-SAT with the Random Geometric Graph: boolean literals are placed uniformly at random or according to a Poisson process in a cube, and for each set of $k$ literals contained in a ball of a given radius, a clause is formed. For $k=2$ we find the exact location of the satisfiability threshold (as either the radius or intensity of the Poisson process varies) and show the threshold is sharp; for $k\ge 3$ we give bounds on the threshold that differ by a constant factor; and for one of the two models we prove that the threshold is in fact sharp for all $k \ge 2$.   

We give a characterization of vertex-monotone properties with sharp thresholds in a Poisson random geometric graph or hypergraph.  As an application we show that a geometric model of random $k$-SAT exhibits a sharp threshold for satisfiability.

\end{abstract}

\section{Introduction}

A property $H$ of a discrete random structure is said to exhibit a \textit{sharp threshold} with respect to a parameter $p$ if there exists a $p_c=p_c(n)$ so that for every $\eps >0$,  for $p> (1+\eps) p_c$, $H$ holds with probability $1-o(1)$ and for $p< (1-\eps) p_c$, $H$ holds with probability $o(1)$.  The classic sharp thresholds in the \ER random graph $G(n,p)$ are the threshold for connectivity at $p_c= \log n /n$ and the threshold for a giant component at $p_c = 1/n$, see~\cite{alon2004probabilistic}.  A property that does not exhibit a sharp threshold is that of containing a triangle: for any $c \in (0,\infty)$, when $p=c/n$ the probability that $G(n,p)$ contains a triangle is strictly bounded away from $0$ and $1$.  %These thresholds differ in their complexity and the way in which the transition occurs.  Both properties are global phenomena, but while the transition to connectivity is locally determined by the disappearance of the final isolated vertex, the emergence of a giant component is not local.  

In addition to much investigation of the threshold location and behavior of specific properties of random graphs, there have been a series of papers proving general threshold theorems.  The first such result was by Bollob{\'a}s and Thomason \cite{bollobas1987threshold} showing that any monotone property $A$ (a property closed under adding additional edges) has a threshold function: a $p^*(n)$ so that for $p \gg p^*$, $G(n,p)$ has property $A$ with probability tending to $1$, and for $p \ll p^*$, $G(n,p)$ has property $A$ with probability tending to $0$.  Subsequently, Friedgut and Kalai \cite{friedgut1996every} showed that every monotone property has a threshold width bounded by $O(\log^{-1} n)$: for any $\eps >0$, if $G(n,p)$ has property $A$ with probability $\eps$, then $G(n, p+C(\eps)/\log n)$ has property $A$ with probability at least $1- \eps$. Bourgain and Kalai \cite{bourgain1998threshold} improved this upper bound to $O(\log^{\del -2} n)$ for any $\del >0$.  Nevertheless, these theorems do not imply a sharp threshold in the sense defined above unless the critical probability for the property is sufficiently high.

Friedgut \cite{friedgut1999sharp} gave a characterization of all monotones properties of random graphs that exhibit a sharp threshold: essentially they are properties that cannot be approximated by the property of containing a subgraph from a list of constant-size subgraphs. In other words, properties with coarse thresholds are all similar to the property of containing a triangle.  Friedgut used his general theorem to prove that the satisfiability of a random $k$-SAT formula exhibits a sharp threshold, and then Achlioptas and Friedgut \cite{achlioptas1999sharp} used it to prove that the property of being $k$-colorable has a sharp threshold in $G(n,p)$. These properties had resisted previous analysis in part because of their complexity: determining the satisfiability of a $k$-SAT formula or the $k$-colorability of a graph are both NP-hard problems.  In contrast, $2$-SAT has a polynomial-time algorithm in the worst-case, and the threshold location \cite{chvatal1992mick} and width \cite{bollobas2001scaling} of a random $2$-SAT formula are both well understood. 

%A third phase transition,  studied in theoretical computer science, occurs in the random satisfiability model.  A random $k$-SAT formula of $m$ clauses on $n$ variables is the conjunction of $m$ disjunctions of $k$ literals (variables or their negations) chosen uniformly from all such formulae.  When $k=2$, the threshold behaves much like the emergence of the giant component in a random graph: the threshold is at $m=n$ clauses and the scaling widow (the number of additional clauses to raise the probability of unsatisfiability from $\eps$ to $1 - \eps$) is of order $n^{2/3}$.  For $k\ge 3$, however, the situation is more complex. The underlying $k$-SAT problem is NP-hard in the worst case and so there is no known simple characterization of satisfiability. Upper and lower bounds on the location of the threshold are known, and for large $k$ these bounds become close~\cite{achlioptas2003threshold,coja2012going}.  The threshold is sharp in a non-uniform sense, due to Friedgut~\cite{friedgut1999sharp}: there exists a function $r_k(n)$ so that for $m > (r_k +\eps) n$, the random formula is UNSAT whp and for $m < (r_k - \eps) n$ the formula is SAT whp.  It remains an open question whether the limit $\lim_{n \to \infty}r_k(n)$ exists.

In this paper we prove a general sharp threshold theorem for the \emph{Random Geometric Graph} (RGG). The typical model of a RGG involves placing $n$ points uniformly at random (or according to a Poisson process of intensity $n$) in $[0,1]^d$ and joining any two points at distance less than $r(n)$ by an edge.   Unlike the edges in $G(n,p)$, the edges in the RGG are not independent. The RGG exhibits thresholds for some of the same properties as the \ER random graph. There is a unique giant component whose appearance occurs sharply at the threshold radius $r_c  = \lambda_c n^{-1/d}$~\cite{penrose:book}. The exact value of the constant $\lambda_c$ is not known, but numerical simulations for $d=2$ indicate $\lambda_c \approx 1.44$~\cite{RintoulTorquato1997} and bounds are given in~\cite{meester-1996-continuum,kong2007analytical}. The RGG also has a sharp threshold for connectivity at $r_c   =  (\log n /(n V_d))^{1/d}$~\cite{gupta-1998-critical, penrose97longest} where $V_d$ is the volume of a unit ball in $\R^d$. %Unlike in $G(n,p)$ the threshold for $k$-colorability is not sharp;  McDiarmid and M\"uller~\cite{mcdiarmid-2011-chromatic} have studied the relationship between the chromatic number and the clique number as $r$ varies. 

The RGG has been extensively studied in fields such as cluster analysis, statistical physics, hypothesis testing, and wireless sensor networks.  One further application of the RGG is modeling data in a high-dimensional space, where the  coordinates of the nodes of the RGG represent the attributes of the data.  The metric imposed by the RGG then depicts the similarity between data elements in the high-dimensional space.  See~\cite{balister2008percolation} or~\cite{penrose:book} for a survey of results on the RGG.

In the RGG, Goel et al. have shown that  every monotone property has a threshold width (in terms of $r$) of $O(\log ^{3/4} n /\sqrt n)$ (for $d=2)$ and $O(\log ^{1/d} n / n^{1/d})$ (for $d \ge 3$)~\cite{goel-2004-sharp}. This implies a sharp threshold in the sense described above when the critical radius of a property is sufficiently large, but not for sparser graphs, and in particular not in the connectivity or giant component regimes.  For one-dimensional RGG's, McColm proved that very monotone property has a threshold function \cite{mccolm2004threshold}, in the sense of Bollob{\'a}s-Thomason.  

We prove a general criteria for sharp thresholds in the Poisson RGG.  As an application, we introduce a geometric model of random $k$-SAT in which literals are placed at random in $[0,1]^d$, and prove that satisfiability exhibits a sharp threshold in this model.  We also identify the location of this threshold in the case $k=2$.  Previously, a model of random $k$-SAT for $k=1,2$ with literals placed on a 2-dimensional lattice was proposed in \cite{schwarz2004percolation}, and in \cite{montanari2007simple} the authors investigate a model of random $k$-XOR-SAT with finite interaction range, a kind of one-dimensional geometry.

 The organization and main contributions of this paper are as follows:
\begin{enumerate}
\item In Section \ref{sec:NotSection}, we introduce notation, define two models of RGG's, and define a sharp threshold in each model.  We then define analogous models of random geometric $k$-SAT. 
\item In Section \ref{sec:sharpThresh} we state our main result: a characterization of vertex-monotone properties with sharp thresholds in the Poisson RGG.  We also state a result on transferring sharp thresholds from the Poisson to fixed-$n$ model.
\item In Section \ref{sec:geomSATsec} we state our results on random geometric $k$-SAT: for all $k\ge 2$, the satisfiability phase transition is sharp in the Poisson model.  For $k=2$, we find the location of this threshold.
\item Section \ref{proofsec} contains the proofs of the sharp threshold lemma and the sharpness of the satisfiability phase transition.
\item Sections \ref{app:density}-\ref{sec:firstmomSec} contain auxiliary results and proofs.

%\item In one model, for all $k\ge 2$, random geometric $k$-SAT undergoes a sharp transition from satisfiability to unsatisfiability.  The proof involves a novel application of Bourgain's criteria for sharp thresholds in random discrete structures.  
%\item For $k=2$, the transition occurs at $m=n$ clauses in both models, as in random $2$-SAT.
%\item For $k\ge 3$, we provide upper and lower bounds on the threshold, in particular showing that the number of clauses at the satisfiability threshold is linear in the number of variables.

\end{enumerate}

%These results show that one fundamental feature of the satisfiability phase transition, sharpness, extends to a natural model in which clauses are dependent, and for $2$-SAT the critical clause density is robust to this change in model.  

 %In the limiting regime when $n \to \infty$, the ratio $\chi_n/\omega_n$ tends to: (i)  $1$ if $nr^d/\log n \to 0$; (ii) $1/\dl$ where $\dl$ is the translational density if $nr^d/\log n \to 0$; (iii) to a  positive constant if $nr^d/\log n$ tends to a positive constant. 

\section{Models and Notation}
\label{sec:NotSection}

%\paragraph{Notation}
We will denote point sets in $[0,1]^d$ by $S, T$ and the graphs, hypergraphs or formulae formed by joining $2$ (or $k$) points that appear in a ball of radius $r$ by $G_S, G_T, F_S, F_T$ respectively.  

We denote graph properties by $A$ and write $G \in A$ if graph $G$ has property $A$.  We say a property $A$ holds `with high probability' or `whp' if $\Pr [G \in A] = 1-o(1)$ as $n \to \infty$.  We write $f(n) \sim g(n)$ if $f(n) = g(n) (1+o(1))$.

We work with two models of random geometric graphs.  $G_d(n, \mu, r)$ is the random graph formed by drawing a point set $S$ according to a Poisson point process of intensity $n \cdot \mu$ on $[0,1]^d$ and then forming $G_S$ by joining any two points at distance $\le r$.  For the hypergraph version of this model, we form a $k$-uniform hyperedge on any set of $k$ points in $S$ that appear in a ball of radius $r$.  If $t>k$ points all appear in one ball of radius $r$, then all $\binom t k$ possible $k$-uniform hyperedges are formed. The second model $G_d(n,r)$ is the random graph drawn by placing $n$ points uniformly and independently at random in $[0,1]^d$ to form $S$, then forming $G_S$ by connecting points at distance $\le r$.  Note that $G_d(n,r)$ has the same distribution as $G_d(n,\mu,r)$ conditioned on $|S|=n$.  

We say a property $A$ has \emph{sharp threshold} in $G_d(n,\mu,r)$ if there exists a function $\mu^*(n), r(n)$ so that for any $\eps >0$,
\begin{enumerate}
\item For $\mu > (1+\eps) \mu^*$, $\Pr[ G_d(n,\mu,r) \in A] = 1-o(1)$.
\item For $\mu < (1-\eps) \mu^*$, $\Pr[ G_d(n,\mu,r) \in A] = o(1)$.
\end{enumerate}

For $G_d(n,r)$ it is more convenient to describe a sharp threshold in terms of the probability that two random points in $[0,1]^d$ form an edge\footnote{For constant dimension $d$, this definition is equivalent to asking for a critical threshold radius $r^*$, but for $d=d(n) \to \infty$, allowing $r$ to increase by a factor $(1+\eps)$ will cause a super-constant factor increase in the number of edges of the graph.}.  We write $r(p)$ for the radius that achieves edge probability $p$.   With this definition, we say that a property $A$ has \emph{sharp threshold} in $G_d(n,r)$ if there exists a function $p^*(n)$ so that for any $\eps >0$,
\begin{enumerate}
\item For $p > (1+\eps) p^*$, $\Pr[ G_d(n,r(p)) \in A] = 1-o(1)$.
\item For $p < (1-\eps) p^*$, $\Pr[ G_d(n,r(p)) \in A] = o(1)$.
\end{enumerate}

For the $k$-SAT problem, we will work with formulae on $n$ boolean variables $x_1, \dots, x_n$.  A \textit{literal} is a variable $x_i$ or its negation $\overline x_i$.  We say a formula $F \in SAT$ if $F$ is satisfiable.  

We define two random geometric distributions over $k$-SAT formulae, $F_k(n,\gamma)$ and $F_k(n,\mu)$:
\begin{itemize}
\item $F_k(n,\gam)$: Randomly place $2n$ points uniformly and independently in $[0,1]^d$ each labeled with the name of a unique literal in $\{ x_1, \dots, x_n, \overline x_1, \dots, \overline x_n\}$.  For any set of $k$ literals that appear in a ball of radius $r= \gamma n^{-1/d}$, form the corresponding $k$-clause and add it to the random formula.  
\item $F_k(n,\mu)$: Draw independent Poisson point processes of intensity $\mu$ on $[0,1]^d$ for each of the $2n$ literals.  For any set of $k$ literals that appear in a ball of radius $r=n^{-1/d}$, add the corresponding clause.
\end{itemize}

Note that $F_k(n,\gam)$ with $\gam=1$ has the same distribution as $F_k(n,\mu)$ conditioned on each literal appearing exactly once.

In this work, we will consider $k, \gam, \mu$ and $d$ fixed with respect to $n$, and take asymptotics as $n \to \infty$.   We use $\ell_\infty$ balls for simplicity in what follows, but all results hold for Euclidean balls as well, with constants involving the volume of the $d$-dimensional unit sphere.

%\paragraph{Remark on the choice of model}
Another natural model to consider would be the following, call it $\tilde F(n, r)$: randomly place $n$ points uniformly and independently in $[0,1]^d$, each labeled with the name of a variable $x_1, \dots x_n$ (instead of  the name of a literal).  Then for each set of $k$ variables appearing in a ball of radius $r$, add a $k$-clause with the signs of the $k$ variables chosen uniformly and independently from the $2^k$ possible choices.  The threshold behavior of satisfiability in $\tilde F(n, r)$ is simpler than in the other two models: the threshold is coarse, and determined locally by large cliques of variables (see Section~\ref{sec:coarse}).

\section{Sharp Thresholds in Random Geometric Graphs}
\label{sec:sharpThresh}

%In a Poissonized random geometric graph, we can vary three parameters: the size of the cube, the intensity of the vertex point process, and the critical radius for an edge. A sharp threshold can be defined with respect to any of these parameters, or with respect to the expected number of edges in the graph.  If the dimension $d$ of the bounding cube is a constant with respect to $n$, then 

The following theorem characterizes vertex-monotone properties with sharp thresholds in the Poissonized random geometric graph $G_d(n,\mu,r)$.  It is an application of Bourgain's theorem in the appendix of Friedgut's paper on sharp thresholds in random graphs \cite{friedgut1999sharp}. 

%\begin{lem}
%\label{GenSharpLemma}
%For every $q \in (0,1)$ and $\eps >0$, there exists a $K(q, \eps)$ so that for any vertex-monotone property $A$ of a $k$-uniform hypergraph so that $\Pr [G_d(n, \mu,r) \in A] =q$, if 
%\[ \mu \cdot \frac{d \Pr [G_d(n, \mu,r) \in A]}{d \mu} \le K(q, \eps)  \]
%then either
%\begin{enumerate}
%\item $\Pr_{G_d(n, \mu,r)} [\exists H \subseteq S: |H| \le K(q, \eps), G_H \in A] \ge \eps$.
%\item  There exists a point set $T$ in $[0,1]^d$ with $| T | \le K(q, \eps)$, $G_T \notin S$ so that 
%\[ \Pr  [G_d(n, \mu,r) \in A | T \subseteq S]  \ge 1- \eps \, . \]
%\end{enumerate}
%\end{lem}

\begin{theorem}
\label{GenSharpLemma}
Let $A$ be a vertex-monotone property of a $k$-uniform hypergraph that does not have a sharp threshold in $G_d(n, \mu, r)$. Then there exists constants $\eps, \del, K >0$ independent of $n$ so that for arbitrarily large $n$ there is an $\alpha \in (\del,1-\del)$ so that either
\begin{enumerate}
\item $\Pr_{G_d(n, \mu,r)} [\exists H \subseteq S: |H| \le K, G_H \in A] \ge \eps$,

or
\item  There exists a point set $T$ in $[0,1]^d$ with $| T | \le K$, $G_T \notin A$ so that 
\[ \Pr  [G_d(n, \mu,r) \in A | T \subseteq S]  \ge \alpha+ \eps \, . \]
\end{enumerate}
with $\mu$ chosen so that $\Pr_{G_d(n, \mu,r)}[A]=\alpha$.
\end{theorem}

In other words, if a property does not have a sharp threshold, then either there is a constant probability that a constant-size witness of $A$ exists in the RGG or there is a point set of constant size in $[0,1]^d$ that by itself does not have property $A$, but by conditioning on the presence of these points significantly raises the probability of $A$ in the RGG. To prove that a property has a sharp threshold, we rule out both of these possibilities.

We can connect sharp thresholds in $G_d(n, \mu, r)$ with those in $G_d(n,r)$.  In particular, if the threshold intensity $\mu^*(n)$ has a limit, then there is a sharp threshold edge probability $p^*$ in $G_d(n,r)$, and it too is uniform in $n$, up to a technical condition on the form of the threshold density.

\begin{prop}
\label{prop:muRsharp}
Let $q(n) = a \log^b(n) n^{-c}$ be a decreasing function of $n$ for constants $a,b,c$.  Suppose a vertex and edge-monotone property $A$ has a uniform sharp threshold in $G_d(n,\mu,r)$: there exists a constant $\mu^*$, independent of $n$, so that
\begin{enumerate}
\item For $\mu > (1+\eps) \mu^*$, $\Pr[ G_d(n,\mu,r(q(n))) \in A] = 1-o(1)$
\item For $\mu < (1-\eps) \mu^*$, $\Pr[ G_d(n,\mu,r(q(n))) \in A] = o(1)$, 
\end{enumerate}
then $A$ has a sharp threshold in $G_d(n,r)$ in a uniform sense: there exists a $t^*$, independent of $n$, so that
\begin{enumerate}
\item For $p > (1+\eps) t^* q(n)$, $\Pr[ G_d(n,r(p)) \in A] = 1-o(1)$.
\item For $p < (1-\eps) t^* q(n)$, $\Pr[ G_d(n,r(p)) \in A] = o(1)$.
\end{enumerate}
\end{prop}

Here we think of $q(n)$ as a typical threshold function, for example: $1/n, c/n^2, \log^2 n/n,$ etc.  The technical condition on $q$ is required to rule out properties whose definition depends non-uniformly on $n$, eg. for small $n$, $A$ is  the property of containing a triangle, while for large $n$, it is the property of containing an edge.

We conjecture that in fact all edge-monotone properties in $G_d(n,r)$ can be characterized similarly:

\begin{conj}
For every edge-monotone property $A$ with a coarse threshold in $G_d(n,r(p))$ with respect to $p$, there are constants  $K, \eps, \del>0$ so that for large $n$, $\alpha \in (\del, 1-\del)$, and $p$ chosen so that $\Pr[G_d(n,r(p))\in A] = \alpha$,  either
\begin{enumerate}
\item $\Pr_{G_d(n, r(p))} [\exists H \subseteq S: |H| \le K, G_H \in A] \ge \eps$,

or
\item  There exists a point set $T$ in $[0,1]^d$ with $| T | \le K$, $G_T \notin A$ so that 
\[ \Pr  [G_d(n, r(p)) \in A | T \subseteq S]  \ge \alpha+ \eps \, . \]
\end{enumerate}
\end{conj}

\section{Random geometric $k$-SAT}
\label{sec:geomSATsec}

As an application of Theorem \ref{GenSharpLemma}, we prove that in the $F_k(n,\mu)$ model, the threshold for satisfiability is sharp:

\begin{theorem}
\label{sharpnessthm}
For all $k$, there exists a function $\mu_k^*(n)$ so that for every $\eps >0$,
\begin{enumerate}
\item For $\mu < \mu_k^*(n) - \eps$, $F_k(n, \mu) \in SAT$ whp. 
\item For $\mu > \mu_k^*(n) + \eps$, $F_k(n, \mu) \notin SAT$ whp.
\end{enumerate}
\end{theorem}

Next, for $k=2$ we determine the exact location of the satisfiability threshold in both models:

\begin{theorem}
\label{2satthm}
For any $\eps>0$,
\begin{enumerate}
\item If $\gam < 2^{-(1+1/d)} - \eps$, then whp $F_2(n,\gam) \in SAT$.  If $\gam > 2^{-(1+1/d)} + \eps$, then whp $F_2(n,\gam) \notin SAT$.  
\item If $\mu <  2^{-(d+1)/2} - \eps$, then whp $F_2(n, \mu) \in SAT$.  If $\mu > 2^{-(d+1)/2}+ \eps$, then whp $F_2(n, \mu) \notin SAT$.
\end{enumerate}
\end{theorem}
Note that from Proposition~\ref{clausesProp} in Section~\ref{app:density}, both thresholds occur at $m=n$ clauses, matching the threshold for random $2$-SAT.

\section{Proofs}
\label{proofsec}

\subsection{Proof of Theorem \ref{GenSharpLemma}}

To prove Theorem \ref{GenSharpLemma}, we discretize $[0,1]^d$ and place points independently at each gridpoint with a given probability.  We apply Bourgain's theorem in a dual fashion, to the product space over positioned points instead of the product space of edges as in $G(n,p)$.  We then show that with a fine enough discretization, the graph formed in the discrete model is identical to the graph formed in the Poisson model with high probability. 

We will prove the theorem for labeled $k$-uniform hypergraphs, where the label set is $\{1, 2, \dots, L(n)\}$ and the dimension $d= d(n)$ may be constant or tend to infinity with $n$.  Points with label $i$ will appear in $[0,1]^d$ according to a Poisson point process of intensity $n\mu / L$, with all labels appearing independently (thus the union of all labeled points is itself a Poisson point process of intensity $n\mu$).  For a random geometric graph we can specialize to $k=2$ with a single label.  For random geometric $k$-SAT, the label set will have size $2n$, one label for each literal.  

Place $N^d$ grid points onto $[0,1]^d$ where $N = 16^{d} n^3$ so that gridpoint $(i_1, \dots, i_d)$ is located at $( (i_1-1/2)/N, \dots  (i_d-1/2)/N)$ and each $i_j$ ranges over $\{1, \dots N\}$. To that gridpoint, assign the region $ A_{i_1, \dots, i_d}= ( (i_1-1)/N, i/N] \times \cdots \times ( (i_d-1)/N, i_d/N]$. At each grid point, let each of the $L$ possible labels appear independently with probability $p = \mu n / LN^{d}$ (more than one label can appear at a single grid point).  For every set of $k$ labeled points that appear in a ball of radius $r$ (in $l_2$ or $l_\infty$ distance, depending on the model), include the corresponding hyperedge in the hypergraph.  The following proposition allows us to transfer results from the discrete model to the continuous model:

\begin{prop}
\label{discProp}
There is a coupling of the discrete and continuous model so that with probability $1 - o(1)$, the labeled hypergraph generated by each is identical.
\end{prop}

\begin{proof}

%Consider the continuous model: place points labeled with each of the $2n$ literals in $[0,1]^d$ according to independent Poisson point processes of intensity $\mu$.  If $k$ literals appear in a box ($\ell_\infty$-ball) of side length $n^{-1/d}$, form a clause with those $k$ literals.

%We construct the discrete model via a coupling from the continuous model in such a way that with probability $1- o(1)$, the formulae created by the two processes are the same.

We couple as follows: If at least one point with label $l$ falls in the region $A_{i_1, \dots, i_d}$ in the continuous model, let the label $l$ be present on gridpoint $(i_1, \dots, i_d)$ in the discrete model.  If no point with label $l$ falls in $ A_{i_1, \dots, i_d}$ in the continuous model, then flip an independent coin that is heads with probability 
\[ e^{\mu n/ L N^d}\cdot (\mu n/ LN^d - (1- e^{-\mu n/ L N^d})). \]
If the coin is heads, let $l$ be present at $(i_1, \dots, i_d)$.  

The following facts suffice to prove the proposition: 
\begin{itemize}
\item The coupling is faithful: the probability that gridpoint $(i,j)$ has a point with label $l$ is: 
\[ 1 - e^{-\mu n/LN^d} + e^{-\mu n/L N^d} \cdot e^{\mu n/LN^d}\cdot (\mu n/ L N^d - (1- e^{-\mu n/L N^d})) = \mu n /L N^d \] 
and all gridpoints and literals are independent by construction.
\item With probability $1-o(1)$ no coins come up heads: i.e. no extra labeled points appear in the discrete model.  The probability of heads for a single coin is $O((\mu n/ L N^d)^2)$, and there are at most $L N^d$ coins flipped. By the union bound whp no heads are flipped.
\item With probability $1-o(1)$ no two copies of any one label
appear in the same $A_{i_1, \dots, i_d}$.  The probability that label $l$ appears at least twice in a fixed $A_{i_1, \dots, i_d}$ is $O((\mu n/LN^d)^2)$.  There are $N^d$ such boxes and $L$ labels, so again whp no region contains more than one.
\item With probability $1-o(1)$ no hyperedges disappear and no new hyperedges appear, moving from the continuous to the discrete model.  In the coupling a point moves by at most $1/2N$ in each coordinate.  For $l_1, l_2, l_\infty$ norms this means the point moves at most $d/2N$ with respect to the norm.  For a hyperedge to appear or disappear due to this movement, two points would need to begin at a distance $x \in [r - d/N, r + d/N]$. For a given pair of points uniformly distributed in $[0,1]^d$, this occurs with probability that depends on the norm, but is bounded by $4^{d+1}d r/N$.  Since the total number of points has a $\text{Poisson}(n \mu)$ distribution, we can condition, and whp have at most $2n \mu$ points.  Taking the union bound over $\Theta (n^2)$ pairs of points gives a failure probability of $O(n^2 4^{d+1}dr/N) = o(1)$, from our choice of $N$ and using the fact that $r \le d$ and $d^2 \le 4^d$.
\end{itemize}

\end{proof}

To complete the proof of Theorem \ref{GenSharpLemma}, we apply the following theorem from Bourgain's appendix to Friedgut's work \cite{friedgut1999sharp}.  Bourgain's theorem gives a criteria for a monotone property on a product measure over the Hamming cube to have a sharp threshold, as opposed to Friedgut's result which applies only to random graphs and hypergraphs.

Consider a random subset $S\subseteq [n]$ with $i\in S$ with probability $p$, independently for all $1 \le  i \le n$.  Let $A$ be a monotone property of subsets of $[n]$.  (In the case of random graphs $n = \binom N 2$ and $S$ is the set of present edges, $A$ might be the property of having a triangle or connectedness.)

\begin{theorem*}[Bourgain \cite{friedgut1999sharp}]
 Assume that $\Pr_{p}[A] =\alpha \in (0,1)$, $p\cdot d \Pr_p(A)/ d p \le C$ and $p = o(1)$.  Then there exists $\del(C, \alpha)>0$ so that either 
 
(1) the probability that $S$ contains a subset $H$ of constant size with $H \in A$ is greater than $\del$.

or 

(2) there exists a constant-sized subset (e.g. a subgraph in $G(n,p)$) $H \notin A$ so that $\Pr_p[Q | H \subseteq S] > \alpha + \del$.  (I.e. conditioning on the appearance of this constant sized subset increases the probability of the property significantly).   
\end{theorem*}

We apply this theorem directly to the discrete model above, with the product space $\{0,1\}^{LN^d }$ and $p=\mu n /L N^d$.  A vertex-monotone property on random geometric graphs becomes a monotone property in this hypercube.  Bourgain's theorem is applied as follows: if a property $A$ does not have a sharp threshold, then by the mean value theorem there must be some $\mu$ so that $\Pr_\mu(A) $ is bounded away from $0$ and $1$, and  $ \mu \cdot d \Pr_\mu(A)/ d \mu \le C$, for some constant $C$.  %So if we can show that neither condition (1) nor (2) holds, then there cannot be such a $p$, and so the threshold must be sharp instead of coarse.  
Then Bourgain's theorem asserts that either condition (1) or (2) must hold. The two conditions are equivalent in the discrete and continuous model since the graphs generated are identical with probability $1-o(1)$. 

%In other words, in the discrete model, we define $\mu^*(n)$ so that $\Pr_\mu^*(A) = 1/2$. We then say that a property $A$ has a coarse threshold if for any $\eps >0$, $\Pr_{(1+\eps)\mu^*}(A) \nrightarrow 1$ as $n \to \infty$. Since $\Pr_\mu(A)$ differs by a $o(1)$ additive factor in the discrete and continuous models by Proposition \ref{discProp}, Lemma \ref{GenSharpLemma} is proved.

\subsection{Proof of Proposition \ref{prop:muRsharp}}

Let $t^* = (\mu^*)^c$.  Fix $\eps >0$. 

 First assume $p> (1+\eps) t^* q(n)$, and let $N = \frac{1}{\mu^*} (1 + \eps/2)^{-c} n$. The conditions of Proposition \ref{prop:muRsharp} say that $\Pr[G_d(N,\mu^*(1+\eps/2)^{c/2}, r(q(N))\in A] = 1- o(1)$.  From the concentration of a Poisson, with probability $1-o(1)$, the number of points drawn in $G_d(N,\mu^*(1+\eps/2)^{c/2}, r(q(N))$ is bounded above by $n$.  We also have 
 \begin{align*}
p &> (1+\eps) (\mu^*)^c q(n)  \\
&=  (1+\eps) (\mu^*)^c \frac{a \log^b n}{n^c} \\
&\ge \frac{a (1+\eps/2) \log^b (n/(\mu^* (1+\eps/2)^{-c}))}{(n/\mu^*)^c} \\
&= q(N)
\end{align*}
Since $A$ is both vertex monotone and edge monotone, we have $\Pr[ G_d(n,r(p))\in A] = 1-o(1)$.  

Next assume $p< (1-\eps) t^* q(n)$, and let $N = \frac{1}{\mu^*} (1 - \eps/2)^{-c} n$. The conditions say that $\Pr[G_d(N,\mu^*(1-\eps/2)^{c/2}, r(q(N))\in A] =  o(1)$.  With probability $1-o(1)$, the number of points drawn in $G_d(N,\mu^*(1-\eps/2)^{c/2}, r(q(N))$ is bounded below by $n$, and 
 \begin{align*}
p &< (1-\eps) (\mu^*)^c q(n)  \\
&=  (1-\eps) (\mu^*)^c \frac{a \log^b n}{n^c} \\
&\le \frac{a (1-\eps/2) \log^b (n/(\mu^* (1-\eps/2)^{-c}))}{(n/\mu^*)^c} \\
&= q(N)
\end{align*}
And again since $A$ is both vertex monotone and edge monotone, $\Pr[ G_d(n,r(p))\in A] = o(1)$.

\subsection{$k$-SAT proofs}

\subsection*{Proof of Theorem \ref{sharpnessthm} }

%We will work on the probability space $\{0,1\}^{2n N^d}$ with the $p$-biased product measure, with one coordinate for every pair of literal and grid point, and for $p =\mu/N^d$.  To prove that unsatisfiability has a sharp threshold, we will assume that the threshold is actually coarse, i.e. $p\cdot d \Pr_p(\unsat)/ d p \le C$, where $p$ is chosen so that $\Pr_p[\unsat] =1/2$. We then derive a contradiction.  

To prove Theorem \ref{sharpnessthm}, we will assume that the threshold is coarse: i.e., there is some $\alpha \in (0,1)$ so that $\Pr_\mu(\unsat)= \alpha$, for which $\mu\cdot d \Pr_\mu(\unsat)/ d \mu \le C$. It then suffices to rule out both possibilities in Theorem \ref{GenSharpLemma} to derive a contradiction.   We will show:  (1) whp there is no constant-sized set of positioned literals that is by itself unsatisfiable and (2) there is no constant-sized satisfiable `booster', one that boosts the unsatisfiability probability from $\alpha$ to $\alpha +\eps$ when conditioned on.  Using Proposition \ref{ksatthm} (Section~\ref{sec:firstmomSec}) we can assume that $\mu$ is a constant bounded from above and away from $0$ independent of $n$.

\paragraph{Notation:}
We will denote by $F_H$ the $k$-SAT formula generated by a set of positioned literals $H \subset [0,1]^d$.  Let $G_\mu \subset [0,1]^d$ be a random set of positioned literals chosen according to $2n$ independent Poisson processes of intensity $\mu$, one for each of the $2n$ literals: i.e. $F_k(n,\mu)$ has the distribution $F_{G_{\mu}}$.  We will use $l_\infty$ distance to simplify calculations, but everything holds for $l_2$ or $l_1$ distance as well, with $\alpha_d$, the volume of the $d$-dimensional unit ball replacing $2^d$ in the calculations below.

\textit{Condition 1:} For any constant $R$, we show that whp there is no set of $R$ positioned literals that form an unsatisfiable formula.  We will use the \textit{implication graph} of a $2$-SAT formula: the directed graph on $2n$ vertices, each representing a literal in the formula, in which $l_1 \to l_2$ if the clause $(l_2 \vee \overline l_1)$ is in the formula.  
 A \textit{bicycle} (see eg. \cite{chvatal1992mick}) of length $L$ in a 2-SAT formula is a sequence of clauses
\[ (u,w_1), (\overline w_1, w_2), (\overline w_2, w_3), \dots, (\overline w_L, v) \]
where the $w_i$'s are literals of distinct variables and $u, v \in \{w_1, \dots, w_L \} \cup \{\overline w_1, \dots, \overline w_L \}$.  A 2-SAT formula is satisfiable if it does not contain a bicycle.
Let $Y_L$ be the number of bicycles of length $L$ in $F_{G_\mu}$. Then 
\begin{equation}
\E Y_L \le    n^L 2^L (2L)^2 \Pr \left [ (\overline u, w_1),  (w_L, v) \in F_{G_\mu}  \wedge \bigwedge_{i=1}^{L-1} (\overline w_i, w_{i+1}) \in F_{G_\mu} \right ]\,.
\end{equation}

\begin{claim}
\label{muClaim}
%Let $F_{G_p}$ be a 2-SAT formula where $G_p \subseteq \{0,1\}^{2n N^d}$.
%
The probability that a specified bicycle of length $L$ appears in $F_{G_\mu}$ satisfies: 
\[ \Pr \left [ (\overline u, w_1),  (w_L, v) \in F_{G_\mu}  \wedge \bigwedge_{i=1}^{L-1} (\overline w_i, w_{i+1}) \in F_{G_\mu} \right ]  \le \frac {\mu^2 + 3 \mu +1 } {\mu^2} \left( \frac{2^d \mu^2 }{n} \right) ^{L+1} \,, 
\]
where $w_i$'s are literals of distinct variables and $u, v \in \{w_1, \dots, w_L \} \cup \{\overline w_1, \dots, \overline w_L \}$.
\end{claim}

\begin{proof}
The literals in the above event are not all distinct, and so the clauses are not all independent.  There may be two literals that are repeated as $\overline u$ and $v$, and perhaps $\overline u = v$.  We consider three different cases for the overlapping clauses:

Case 1: $u \ne v$, $(u,v) \ne (\overline w_i, w_{i+1})$ for any $i$.

Say $u= w_i$ and $v= w_j$, though the argument will be the same if either or both is a negation.  For $k \ne i-1$ or $j-1$, the clauses $\overline w_k, w_{k+1}$ are independent of all other clauses in the bicycle.  Each has probability of appearing $\sim 2^d \mu^2/n$ for our choice of $\mu$.  Now consider the pairs of clauses $\{(u=w_i, w_1),(w_{i-1},w_i)\}$ and $\{(w_L,v=w_j), (w_{j-1},w_j)\}$.  The clauses within each pair are not independent, but the pairs are independent of each other.  Both pairs are of the form $(l_1, l_2),(l_1, l_3)$ for distinct literals $l_1,l_2,l_3$.  Conditioning on the number of appearances of $l_1$, we have
\begin{align}
\label{wedgeprob}
\nonumber
\Pr[ (l_1, l_2), (l_1, l_3) \in F] &\sim \sum_{j=0}^\infty \frac{ e^{-\mu} \mu^j}{j!} \left( 1- e^{-{2^d\mu j/n}} \right )^2 \\
\nonumber
&\sim  \sum_{j=0}^\infty \frac{ e^{-\mu} \mu^j}{j!}\frac{2^{2d} \mu^2 j^2}{n^2} \\
&= \frac{2^{2d} \mu^2}{n^2 } \left( \mu + \mu^2 \right) \,.
\end{align}
%

%In this case $V=V_d \left(\gm n^{-1/d} \right) = \al_d n^{-1}$.

%\textbf{Euclidean norm $\ell_2$ case: }

%\begin{align}
%\label{wedgeprobE}
%\nonumber
%\Pr[ (l_1, l_2), (l_1, l_3) \in F] &\sim \sum_{j=0}^\infty \frac{ e^{-\mu} \mu^j}{j!} \left(1 - \exp\left(-j \mu V\right)\right)^{2} \\
%\nonumber
%& \sim  \sum_{j=0}^\infty \frac{ e^{-\mu} \mu^j}{j!}\frac{ j^2 \mu^2 V^2}{n^2} \\
%&= V^2 \mu^3  \left( \mu + 1\right ) \\
%&= \frac{\al_d^2} {n^2} \mu^3 \left(\mu + 1\right) \,.
%\end{align}
%
All together, with the $L-3$ independent clauses, this gives that a bicycle of this type appears with probability at most
  \[  \left(  \frac{2^{2d} \mu^3(\mu +1)}{n^2} \right)^2 \left ( \frac{2^d \mu^2}{n}  \right) ^{L-3}  =   \frac{(\mu+1)^2}{\mu^2} \left( \frac{2^d \mu^2 }{n} \right) ^{L+1} \,. \]
  
%\textbf{Euclidean norm $\ell_2$ case: }
 % \[  \left(  \frac{\al_d^2 \mu^3(\mu +1)}{n^2} \right)^2 \left ( ???\frac{\al_d \mu^2}{n}  \right) ^{L-3}  =   \frac{(\mu+1)^2}{\mu^2} \left( \frac{2^d \mu^2 }{n} \right) ^{L+1} \,. \]

Case 2: $u \ne v$, $(u,v) = (\overline w_i, w_{i+1})$ for some $i$.

For $k\ne i$, the clauses $(\overline w_k, w_{k+1})$ are independent of the other clauses in the bicycle.  What remains is the triple $\{(u=\overline w_i,w_1),(\overline w_i, w_{i+1}),(w_L, w_{i+1})\}$.  (The argument is the same if $u=w_{i+1}$ and $v= \overline w_i$). This triple is of the form $(l_1,l_2),(l_1,l_3),(l_4,l_3)$.  We calculate the probability such a triple appears by conditioning on the number of appearances of $l_1$ and $l_3$:
\begin{align*}
\Pr[ (l_1, l_2), (l_1, l_3),(l_4,l_3) \in F] &\sim \sum_{j=0}^\infty \sum_{k=0}^\infty \frac{ e^{-\mu} \mu^j}{j!} \frac{ e^{-\mu} \mu^k}{k!}  \frac{2^{3d} j^2k^2 \mu^2}{n^3}
\end{align*}
and so 
\begin{align*}
\Pr[ (l_1, l_2), (l_1, l_3),(l_4,l_3) \in F] &\sim \frac{\mu^2}{n^3} \sum_{j=0}^\infty \sum_{k=0}^\infty \frac{ e^{-\mu} \mu^j}{j!} \frac{ e^{-\mu} \mu^k}{k!} j^2 k^2 \\
&= \frac{2^{3d} \mu^2}{n^3} ( \mu+ \mu^2)^2 \\
&= \frac{2^{3d} \mu^4 (\mu+1)^2}{n^3} \,.
\end{align*}
Again all together the probability of the particular bicycle appearing is at most
\[  \frac{2^{3d} \mu^4 (\mu+1)^2}{n^3}  \left ( \frac{2^d \mu^2}{n}  \right) ^{L-2}  =  \frac {(\mu+1)^2} {\mu^2} \left( \frac{2^d \mu^2 }{n} \right) ^{L+1}  . \]
 
%\textbf{Euclidean norm $\ell_2$ case: }

%\begin{align*}
%\Pr[ (l_1, l_2), (l_1, l_3),(l_4,l_3) \in F] &\sim \sum_{j=0}^\infty \sum_{k=0}^\infty \frac{ e^{-\mu} \mu^j}{j!} \frac{ e^{-\mu} \mu^k}{k!}  
%V^3 j^2k^2 \mu^2
%\frac{2^{3d} j^2k^2 \mu^2}{n^3}
%\end{align*}
%and so 
%\begin{align*}
%\Pr[ (l_1, l_2), (l_1, l_3),(l_4,l_3) \in F] &\sim \frac{\mu^2}{n^3} \sum_{j=0}^\infty \sum_{k=0}^\infty \frac{ e^{-\mu} \mu^j}{j!} \frac{ e^{-\mu} \mu^k}{k!} j^2 k^2 \\
%&= \frac{2^{3d} \mu^2}{n^3} ( \mu+ \mu^2)^2 \\
%&= V^3 \mu^2 ( \mu+ \mu^2)^2 \\
%&= V^3 \mu^4 ( \mu+ 1)^2 \\
%&= \frac{\al_d^{3} \mu^4 (\mu+1)^2}{n^3} \,.
%\end{align*}
%Again all together the probability of the particular bicycle appearing is at most
%\[  \frac{\al_d^{3} \mu^4 (\mu+1)^2}{n^3}  \left ( \frac{???\al_d \mu^2}{n}  \right) ^{L-2}  =  \frac {(\mu+1)^2} {\mu^2} \left( \frac{2^d \mu^2 }{n} \right) ^{L+1}  . \]
 
% \textbf{Euclidean norm $\ell_2$ case: }
 
%\[  \frac{\al_d^{3} \mu^4 (\mu+1)^2}{n^3}  \left ( \frac{???2^d \mu^2}{n}  \right) ^{L-2}  =  \frac {(\mu+1)^2} {\mu^2} \left( \frac{2^d \mu^2 }{n} \right) ^{L+1}  . \]
 
Case 3: $u=v$.

Say $u=v= w_i$.  (The same will work for $u=v=\overline w_i$).  The clauses $(\overline w_k, w_{k+1})$ for $k\ne i-1$ are again independent of all other clauses in the bicycle.  What remains are the clauses $(u=w_i, w_1), (\overline w_{i-1},w_i), (w_L,v=w_i)$.  This is a triple of the form $(l_1, l_2), (l_1,l_3), (l_1,l_4)$ and we calculate its probability by conditioning on the number of appearances of $l_1$:

\begin{align*}
\Pr[ (l_1, l_2), (l_1, l_3), (l_1,l_4) \in F] &\sim  \sum_{j=0}^\infty \frac{ e^{-\mu} \mu^j}{j!}\frac{2^{3d} \mu^3 j^3}{n^3} \\
&= \frac{ 2^{3d} \mu^3 }{n^3}(\mu^3+3\mu^2+\mu) = \frac{2^{3d} \mu^4 (\mu^2 + 3 \mu +1)}{n^3} \,.
\end{align*}
%
%\textbf{Euclidean norm $\ell_2$ case: }
%
%\begin{align*}
%\Pr[ (l_1, l_2), (l_1, l_3), (l_1,l_4) \in F] &\sim  \sum_{j=0}^\infty \frac{ e^{-\mu} \mu^j}{j!}\frac{V^3 \mu^3 j^3}{n^3} \\
%&= V^3 \mu^3 (\mu^3+3\mu^2+\mu) = \frac{\al_d^{3} \mu^4 (\mu^2 + 3 \mu +1)}{n^3} \,.
%\end{align*}
%
%
So the probability of such a bicycle is at most
\[
\frac{2^{3d} \mu^4 (\mu^2 + 3 \mu +1)}{n^3}   \left ( \frac{\al_d \mu^2}{n}  \right)^{L-2} =  \frac {\mu^2 + 3 \mu +1 } {\mu^2} \left( \frac{2^d \mu^2 }{n} \right) ^{L+1} . \]
The three estimates prove the claim.

%\textbf{Euclidean norm $\ell_2$ case:}
%\[
%\frac{\al_d^{3} \mu^4 (\mu^2 + 3 \mu +1)}{n^3}   \left ( \frac{2^d \mu^2}{n}  \right)^{L-2} =  \frac {\mu^2 + 3 \mu +1 } {\mu^2} \left( \frac{\al_d \mu^2 }{n} \right) ^{L+1} . 
%\]
The three estimates prove the claim.

\end{proof}

Using the claim  and summing from $L=1$ to $R$ yields:
\[ \sum _{L=1}^R \E Y_L \le \sum_{L=1}^R  (2n)^L (2L)^2 \frac {\mu^2 + 3 \mu +1 } {\mu^2} \left( \frac{2^d \mu^2 }{n} \right) ^{L+1}   = O(n^{-1}) \]
for any $\mu, R$ constant with respect to $n$. So whp there is no bicycle in the implication graph of length $\le R$ and thus no set of $R$ literals that form an unsatisfiable formula.  

For $k\ge 3$ consider an arrangement of $R$ literals that yields an unsatisfiable $k$-SAT formula.  The configuration of points would also induce an unsatisfiable $2$-SAT formula since for each $k$-clause, each of the $\binom k 2$ $2$-clauses from the same set of literals would be present, and a satisfying assignment to the $2$-SAT would also satisfy the $k$-SAT formula.  But whp there is no set of $R$ unsatisfiable $2$-SAT literals, and so no set of $R$ unsatisfiable $k$-SAT literals.

\textit{Condition 2:} We want to show that there is no constant-sized set of positioned literals $H$,  so that $F_H$ is satisfiable but conditioning on the presence of $H$ raises the probability of unsatisfiability of $F_{G_\mu}$ from $\alpha$ to $\alpha + \eps$ at the $\mu$ for which $\Pr [F_k(n,\mu)\notin SAT]=\alpha$.  Assume $|H| \le R$.  We will bound the conditional probability:
\[ \Pr[ F_{G_\mu} \notin SAT | H \subseteq G_\mu] \le \Pr[ F_{G_\mu \cup H} \notin SAT] \,.\] 
In other words, we will create a random formula by first placing the literals in $H$ in the cube, then adding each positioned literal independently on top according to a Poisson process of intensity $\mu$, then forming the $k$-SAT formula from the entire set of points.  Note that in the probability on the RHS $H$ is a fixed point set, and $G_\mu$  a random point set that does not depend on $H$. 

We now bound $ \Pr[ F_{G_\mu \cup H} \notin SAT] $.  Let $\mathcal X_H$ be the set of variables of the literals in $H$. By assumption $|\mathcal X_H| \le kR$.  First we show that whp  the subformula of $F_{G_\mu \cup H}$ consisting of clauses entirely from $\mathcal X_H$ is satisfiable.  By assumption, $F_H$ is satisfiable so to create an unsatisfiable subformula on $\mathcal X_H$ we need the addition of $G_\mu$ to add at least one clause with variables entirely in $\mathcal X_H$.  There are two different ways this could happen - either a clause is created entirely with randomly placed literals, or a clause is created with some literals from $H$ and some random literals.  

We bound the expected number of clauses in $F_{G_\mu}$ containing only variables from $\mathcal X_H$, call this $\E Y_{\mathcal X_H, \mu}$, by bounding the number of literals from $\mathcal X_H$ appearing within distance $ n^{-1/d}$  of each other in $G_\mu$:
\[\E Y_{\mathcal X_H, \mu} \le \binom {2kR}{2} \frac{2^d \mu^2}{n} = o(1) \,.\]

Next, we bound the expected number of literals from $\mathcal X_H$ placed by $G_\mu$ within distance $n^{-1/d}$ of a literal in $H$.  The total volume of the cube within distance $n^{-1/d}$ of $H$ is bounded by $2^d k^2R^2 /n$, and so the expected number of literals from $\mathcal X_H$ appearing at random in this region is bounded by $ 2^d k^2R^2 (2kR \mu)/n  = o(1)$.

The remainder of the proof follows the general plan of Section 5 of \cite{friedgut1999sharp}.  We separate the $n$ variables into two sets $\mathcal X_H$ and $\mathcal X_H^c$, and we have shown that whp after the addition of $G_\mu$  there is an assignment to $\mathcal X_H$ that satisfies the subformula of clauses entirely in $\mathcal X_H$, call this assignment $x_H$.  We now show that with probability at least $1- \alpha - \eps /2$, we can extend this assignment on $\mathcal X_H^c$ to satisfy $F_{G_\mu \cup H}$.  The remaining formula consists of two types of clauses: clauses which contain variables from $\mathcal X_H$ (overlapping clauses) and clauses that contain only variables from $\mathcal X_H^c$ (non-overlapping).  With probability at least $1-\alpha$, the set of non-overlapping clauses in $F_{G_p}$ is satisfiable, from the definition of  $\mu$.  We will show that adding the overlapping clauses decreases this probability by at most $\eps /2$.

Step 1: The overlapping clauses created with the addition of $G_\mu$ are dominated (in terms of inducing unsatisfiability) by adding a constant number of independent random unit clauses.

We can assume that $F_H$ is maximal in the sense that it admits exactly one satisfying assignment, $x_H$.  Adding $H$ to $G_\mu$ has two effects: it adds the constraint that $\mathcal X_H = x_H$ and it may create some new clauses involving positioned literals from $H$ and $G_\mu$.  We have shown above that whp these new clauses all contain at least one variable from $\mathcal X_H^c$. 

Consider the following modification of $F_{G_\mu}$: call the set of literals from $\mathcal X_H^c$ that fall within distance $n^{-1/d}$ of a literal from $\mathcal X_H$ (either in $H$ or in $G_\mu$) $L$.  Note that the literals in $L$ are uniformly random over all literals in $\mathcal X_H^c$.  Remove the set  $L$ from  $G_\mu$   to form the random point set $G_\mu^-$.  Create the formula $F_{G_\mu^-}^*$ by forming $k$-clauses according to the usual rules for $G_\mu^-$, but add a unit clause $(l)$ for every literal $l \in L$ that was removed from $G_\mu$.  Critically the $k$-clauses of $F_{G_\mu^-}^*$ are independent of the unit clauses of $F_{G_\mu^-}^*$ since they are formed from points from disjoint regions of the cube.  Note that if there is a satisfying assignment to $F_{G_\mu^-}^*$, then the same assignment satisfies $F_{G_\mu}$.  The inequality goes in the correct way: we progress to a formula which has less probability of being satisfied. 

The expected number of literals from $G_\mu$ that fall within distance $n^{-1/d}$ of a literal in $\mathcal X_ H$ is bounded by $2^d/n \cdot (\mu+1) 2 kR (2n\mu) = 2^{d+2} kR \mu (\mu+1)$, so with probability $1-\eps/4$ the size of $L$ is at most $2^{d+4} kR \mu (\mu+1)/\eps$.

Now consider the random formula $F^\prime $ which is formed by sampling a copy of $F_{G_\mu}$ and adding to it $2^{d+4} kR \mu (\mu+1)/\eps$ independent, uniformly random unit clauses from all $2n$ literals.  With probability $1-o(1)$ this is the same as adding the same number of uniformly random unit clauses chosen from $\mathcal X_H^c$, and $F_{G_\mu}$ stochastically dominates the $k$-clauses of $F_{G_\mu^-}^*$ (formed from a Poisson process on a larger region), so $\Pr[F^\prime \in SAT] \le  \Pr[F_{G_\mu^-}^* \in SAT] + \eps/4  \le  \Pr[F_{G_\mu \cup H}\in SAT]  + \eps /4 + o(1)$.

Step 2: $\Pr[ F^\prime \in SAT] \ge \Pr[ F_{G_\mu} \wedge C_1 \wedge \cdots \wedge C_{\sqrt n} \in SAT]$, where the $C_i$'s are a collection of $\sqrt n$ independent, uniformly random $k$-clauses.  This is Lemma 5.7 from \cite{friedgut1999sharp}.

Step 3: $\Pr[ F_{G_\mu} \wedge C_1 \wedge \cdots \wedge C_{\sqrt n} \in SAT]\ge \Pr[F_{G_\mu \cup G_{\mu_s}} \in SAT]$, where $G_{\mu_s}$ is an independent sprinkling of random positioned literals with intensity $\mu_s = n^{-\del}$ for each of the $2n$ literals.  

We will sprinkle literals independently, adding each literal as a Poisson process of intensity $\mu_s$. Split the cube into $n$ disjoint small cubes with side length $n^{-1/d}$. The probability that a single small cube has at least $k$ sprinkled literals is $\sim (2 \mu_s)^k/k! = 2^k n^{-k \del}/k!$.  The expected number of boxes with $k$ literals is $\Theta(n^{1- k \del})$ and whp there are at least $n^{1-2k\del}$ such boxes.   If we pick one $k$-clause at random from each box that has one, we will get a set of at least $n^{1 - 2k\del}$ uniform and independent random $k$-clauses. Picking $\del = 1/5k$ suffices.

Step 4: Increasing $\mu$ to $\mu^\prime = \mu + \mu_s$ lowers the probability of satisfiability by at most  $Cn^{-\del}=Cn^{-1/5k}$, from the assumption of a coarse threshold (bounded derivative of the probability with respect  to $\mu$, $\mu\cdot d \Pr_\mu(\unsat)/ d \mu \le C$).

All together we have:
\begin{align*}
 \Pr[ F_{G_\mu} \in SAT | H \subseteq G_p] &\ge \Pr[ F_{G_\mu \cup H} \in SAT] \\
 &\ge  \Pr[F_{G_\mu^-}^* \in SAT] +o(1) \\
 &\ge \Pr[F^\prime \in SAT] - \eps /4 +o(1) \\
 &\ge \Pr[ F_{G_\mu} \wedge C_1 \wedge \cdots \wedge C_{\sqrt n} \in SAT] - \eps/4 + o(1)\\
 &\ge \Pr[F_{G_\mu \cup G_{\mu_s}} \in SAT] - \eps/4 +o(1) \\
 &\ge \Pr[F_{G_\mu} \in SAT] - Cn^{- \del} - \eps / 4 + o(1)
\end{align*}
This contradicts condition 2 in Theorem \ref{GenSharpLemma}, leading to the conclusion that the threshold must in fact be sharp.

\section{Clause density}
\label{app:density}

The clause density in each $k$-SAT model is as follows:
\begin{prop}
\label{clausesProp}
The number of clauses in $F_k(n,\gam)$ is  $\frac{ 2^k \gam^{d(k-1)}  k^d }{k!} n + o( n)$ whp. 
The number of clauses in $F_k(n,\mu)$ is $\frac{(2\mu)^k k^d}{k!} n  + o( n)$ whp.
\end{prop}

%\subsection*{Proof of Proposition \ref{clausesProp} }
\begin{proof}
Theorem 3.4 of \cite{penrose:book} states that the subgraph count of any fixed-size graph in the RGG converges to a normal distribution centered around the expectation.  To compute the expectation in our case, note that the probability that $k$ given points, distributed uniformly at random in $[0,1]^d$ lie in an $\ell_\infty$-ball of radius $\gam n^{-1/d}$ is the probability that the smallest and largest of $k$ independent uniform $[0,1]$ random variables differ by at most $\gam n^{-1/d}$ raised to the $d$th power.  This probability can be computed by conditioning on the position of the smallest value:
\begin{align}
\label{kcliqueprob}
%p_k &= \int_0^1 k (1-t)^{k-1} \left ( \frac{r n^{-1/2}}{1-t} \right)^{k-1}   d t  = \frac{k r^{k-1}}{n^{(k-1)/2}} \,.
\nonumber
p_k &= \int_0^1 k (1-t)^{k-1} \min \left\{1, \left ( \frac{\gam n^{-1/d}}{1-t} \right)^{k-1} \right\}  d t  \\
\nonumber
&= k (\gam n^{-1/d})^{k-1} \int_0^{1-\gam n^{-1/d}} d t + k \int_{1-\gam n^{-1/d}}^1 (1-t)^{k-1} d t \\
\nonumber
&= \frac{k r^{k-1}}{n^{(k-1)/d}} \left( 1 -\frac{k-1}{k} \gam n^{-1/d} \right) = \frac{k \gam^{k-1}}{n^{(k-1)/d}} \left(1+o(1)\right) \,.
\end{align}

So in the $F_k(n,\gam)$ model, if $X$ is the number of clauses,
\[ \E X = \binom {2n}{k} p_k^d \sim \frac{ 2^k \gam^{d(k-1)}  k^d }{k!} n \,.\]

In the $F_k(n,\mu)$ model,  the number of points in the cube has a $\poiss(2 \mu n)$ distribution.  Conditioning on $N$, the number of points, we get
\[ \E X= \E \left[ \binom{N}{k} p_k^d  \right ] \sim \frac{(2\mu)^k k^d}{k!} n  \,.  \]
\end{proof}

\section{A coarse threshold for $\tilde F(n, r)$}
\label{sec:coarse}

Here we show that the model $\tilde F(n, r)$ in which variables are placed in $[0,1]^d$ and signs of clauses drawn uniformly at random has a coarse threshold.  

\begin{prop}
\label{coarseprop}
Let $r= \gam n^{-\frac{U(k)}{d(U(k)-1)}}$, where $U(k)$ is the integer function described in Section \ref{proofsec}.  Then
\[ \lim_{n \to \infty} \Pr[\tilde F(n, r) \in SAT] = g(\gam) \]
for a function $g(\gam) \in (0,1)$.  Further, $\lim_{\gam \to 0} g(\gam) =1$ and $\lim_{\gam \to \infty} g(\gam) =0$.  
\end{prop}

%\subsection*{Proof of Proposition \ref{coarseprop}}

\begin{proof}Let $U(k)$ be the minimal number of variables $u$ so that there exists an unsatisfiable $k$-CNF formula on $u$ variables so that no two clauses share the same set of $k$ variables.  

Claim: $U(k) \le (\ln 2 )^{1/(k-1)} (2k)^{k/(k-1)} $.  In particular, $U(k)$ is finite.  

Proof: Let $u\ge  (\ln 2 )^{1/(k-1)} (2k)^{k/(k-1)} $.  Now consider a random formula formed by taking a clause for each of the $\binom u k$ distinct sets of $k$ variables from the set of variables $x_1, \dots x_u$, and then assigning signs uniformly at random.  The expected number of satisfying assignments is:
\[ 2^u (1- 2^{-k})^{\binom u k} < 1\]
for our choice of $u$ (using basic estimates). So there exists some unsatisfiable formula on $u$ variables in which each clauses has a distinct set of variables. 

Now we show that satisfiability undergoes a coarse threshold at $r= n^{-\frac{U(k)}{d(U(k)-1)}}$.   The general idea of the proof is that for $r =\gam n^{-\frac{U(k)}{d(U(k)-1)}}$, the probability that there is a set of $U(k)$ variables in a ball of radius $r$ is bounded away from $0$ and $1$.  The probability that each such set forms an unsatisfiable formula is also bounded away from $0$ and $1$.  We then show that for this choice of $r$, if there is no such set of variables, the formula is satisfiable whp.  

For $r = \gam  n^{-\frac{U(k)}{d(U(k)-1)}}$  the expected number of sets of $U(k)$ variables that form an unsatisfiable formula tends to a constant as $n \to \infty$.  To see this note that the expected number of sets of $U(k)$ variables that fall in a ball of radius $r$ is a constant, and that any such set of variables is unsatisfiable with probability at least $2^{-U(k)}$ from the definition of $U(k)$.  To see that it is at most a constant, note that the expected number of connected components of $U(k)$ variables is constant.  A modification of Theorem 3.4 of \cite{penrose:book} shows that the number of such unsatisfiable sets of variables has a Poisson distribution asymptotically.  The mean of this Poisson random variable tends to $\infty$ as $\gam \to \infty$ and to $0$ as $\gam \to 0$.   Finally, if there is no such set, then the formula is satisfiable whp, since whp the RGG for this radius consists of connected components of size at most $U(k)$.  For a component of size $< U(k)$, there must be a satisfying assignment, by the definition of $U(k)$. 
\end{proof}

\section{Proof of Theorem \ref{2satthm}}

Unlike in the study random 2-SAT, we must account for dependence between clauses in these models.  Some of the calculations and techniques may be of independent interest to those studying sparse RGG's.  The key part of the proof is that while the structures we analyze with the first- and second-moment methods are long connected components in the implication graph of the formula, they are close to being collections of isolated edges in the graph of literals and clauses, and so we are nearly in the case of independent clauses.  In calculating variances, we must account for more dependence and this is what leads to the bulk of the calculations.

\subsection*{Proof of Theorem \ref{2satthm} for $F_k(n, \gam)$ }
%We give the proof for $F_k(n,\gam)$.  %The proof for $F_k(n, \mu)$ is similar.

\textit{Lower bound:} As above, we will count bicycles in the implication graph of a $2$-SAT formula, and show whp there are none, for $\gam < 2^{-(1+1/d)} -\eps$. We treat large cycles and small bicycles separately, as in \cite{cooper2002note}.  

Large ($L \ge K \log n $):   Let $X_L$ be the expected number of directed paths of length $L$ with distinct variables in the implication graph. Then
\begin{align*}
\E X_L &\le n^L 2^L \Pr \left [ \bigwedge_{i=1}^{L-1} (\overline w_i, w_{i+1}) \in F  \right ]  = n^L 2^L \Pr[ (\overline w_1 , w_2) \in F] ^{L-1} 
\end{align*}
since for $i\ne j$ the clauses $(\overline w_i, w_{i+1})$ and $(\overline w_j, w_{j+1})$ are made up of four different literals (though if $j= i+1$ the underlying variables might repeat).  So the listed clauses are independent.  The probability of a given $2$-clause $(l_1,l_2)$ being present is $(2 \gam)^d/n$ so $\E X_L \le (2n)^L \left( \frac{(2 \gam)^d}{n} \right) ^{L-1} = 2n \left (2(2 \gam)^d \right)^{L-1}$.  For $\gam < 2^{-(1+1/d)} -\eps$ and $L > K \log n$ for large enough $K(\eps)$, $\E X_L = o(1)$, so whp there are no long bicycles.
 
Small ($L < K \log n $): we will show that whp there is no bicycle of length $\le K \log n$ when $\gam < 2^{-(1+1/d)} -\eps$.  Let $Y_L$ be the number of bicycles of length $L$.  Then
\begin{align*}
\E Y_L &\le    n^L 2^L (2L)^2 \Pr \left [ (u, w_1),  (\overline w_L, v) \in F  \bigwedge \bigwedge_{i=1}^{L-1} (\overline w_i, w_{i+1}) \in F  \right ]  
\end{align*}
where as above $u, v \in \{w_1, \dots, w_L\}$ or their negations.   Unlike above, the clauses in the event in brackets are not made up of entirely distinct literals.  There may be two literals that are repeated as $\overline u$ and $v$, and perhaps $\overline u = v$.  However, the clauses do form a forest in the graph of literals and clauses, and in the $F_k(n,\gam)$ model, as with edges in the RGG, the appearance of 2-clauses in given forest are independent, and so the probability in brackets is $((2 \gam)^d/n)^{L+1}$. All together we have
\begin{align*}
\sum_{L=2}^{K \log n} \E Y_L &\le \frac{2 L^2}{n} (2 (2 \gam)^d)^{L+1}  = O \left( \frac{\log^3 n}{n}  \right)
\end{align*}
for $\gam < 2^{-(1+1/d)} -\eps$ and so whp there are no short bicycles either.

\textit{Upper Bound:} For the upper bound, we use the first- and second-moment methods on the number of \textit{snakes} of length $\log^2 n$ in the random formula.  A \textit{snake} (see \cite{chvatal1992mick}) of length $s = 2t-1$ is a collection of clauses
\[ (w_t, w_1), (\overline w_1, w_2), (\overline w_2, w_3), \cdots, (\overline w_t, w_{t+1}),  \cdots, (\overline w_s, \overline w_t) \]
where the $w_i$'s are literals corresponding to distinct variables.  Note that a snake is unsatisfiable: choosing either $w_t =T$ or $w_t=F$ leads to a chain of implications resulting in a contradiction.  

The structure of a snake is a forest on the graph of literals, and so the clauses are independent in the $F_k(n, \gam)$ model.  The probability a given $s$-snake is present is $ \sim ( (2\gam)^d /n ) ^{s+1}$. Let $X_s$ be the number of snakes of length $s$.  Then $\E X_s \sim  \binom {n}{s} 2^s s! ( (2 \gam)^d/n)  ^{s+1}$, and for $s = \log^2 n$, 
\begin{equation}
\label{xsexp}
\E X_s \sim  (2n)^s \left( \frac{(2 \gam)^d}{n} \right)^{s+1} = \frac{1}{2n} (2 (2 \gam)^d)^{s+1} \to \infty
\end{equation}
since  $\gam >  2^{-(1+1/d)} + \eps$.  A similar calculation works in the $F_k(n, \mu)$ model when $\mu > 2^{-(d+1)/2} + \eps$. 

Next we prove the following:

\begin{prop}
\label{varprop}
For $s = \log ^2 n$ and $\gam > 2^{-(1+1/d)} + \eps$,
\[ var(X_s) = o( (\E X_s)^2) \]
\end{prop}
From (\ref{xsexp}) and Proposition \ref{varprop} we conclude that $X_s \ge 1$ whp and thus $F_k(n, \gam)$ is unsatisfiable whp.

\textit{Proof of Proposition \ref{varprop}}: Let $A$ and $B$ be two $s$-snakes, and say $A$ and $B$ overlap in $i$ clauses and their union forms $j$ cycles in the graph of literals and clauses.  Then the probability that snakes $A$ and $B$ are both present in the random formula is bounded above by $((2 \gam)^d/n)^{ 2(s+1)-i-j  } $ since the union of the two snakes has $2(s+1) -i$ clauses and we can remove $j$ clauses to form a forest, then use the fact that clauses in a forest are independent.  Now consider an arbitrary snake $A$ and a random snake $B$.  We bound the probability that $A$ and $B$ overlap in a given way:

\begin{claim}
\label{overlapclaim1}
Fix an $s$-snake $A$ and choose $B$ uniformly at random from all $s$-snakes. Let $p_{ij}$ be the probability that $A$ and $B$ share $i$ clauses and their union forms $j$ cycles. Then

\[ p_{0,j} \le 2^{10} s^{10}  \left(\frac{1}{2n-2s} \right )^{j+1} , \]
for $1\le i \le t-1$ 
\[ p_{i,j} \le 2^{10} s^{15}  \left (\frac{1}{2n -2s} \right) ^{i+j+1}   \,, \]
and for $i \ge t$,
\[ p_{i,j} \le2^{5} s^{4}  \left (\frac{1}{2n -2s} \right) ^{i+j-6}   \,.  \] 
\end{claim}

\begin{proof}
This claim is similar to (8) and (9) in \cite{chvatal1992mick}, but we need more precision to count overlapping literals as well as clauses.  Denote the clauses of $A$ by 
\[ (x_1, x_t), (\overline x_1, x_2), (\overline x_2, x_3), \dots, (\overline x_{s-1}, x_s), (\overline x_s, \overline x_t)\]
 and the clauses of $B$ by 
 \[ (w_1, w_t), (\overline w_1, w_2), (\overline w_2, w_3), \dots, (\overline w_{s-1}, w_s), (\overline w_s, \overline w_t). \]
   We will call each of the four two-paths a \textit{hinge}.

First we consider the case $i=0$.  The union of $A$ and $B$ has at most 4 cycles involving a hinge from $A$ or $B$, and these cycles can have length as small as 3 and as few as 2 overlapping literals.  All other cycles in the union must be even cycles of length at least 4 with at least 4 overlapping literals.   Let $r$ be the number of cycles in the union with a hinge, and $j-r$ the number of additional cycles.  One bound on the probability is to simply count the number of variables that must overlap, which is $(4(j -r) +2r)/2 = 2j-r$ (dividing the number of overlapping literals by $2$).   The probability of $l$ variables overlapping between $A$ and $B$ is bounded above by $( s^2/(n-s))^{l}$, and so the claim follows if $r<j$, and using the fact that $r \le 4$.  Now if $r=j$, i.e. all cycles in the union involve a hinge, we show that the number of overlapping variables is at least $j+1$. Clearly this is true if $j=1$: two variables must overlap in a cycle. 

 Now if $j=2$, and both hinges come from $A$ or both from $B$, the variables that need to be joined to form a cycle are different in the two hinges: $\{ x_{t-1} , x_1\}$ and $\{x_{t+1},  x_s\}$, so we have at least $4$ overlapping variables.  Now if one hinge is from $A$ and the other from $B$, we can check that at most one variable can overlap, unless the cycle is actually the same.  E.g., consider the hinge $(\overline w_{t-1}, w_t, w_1)$ in $B$ and $(x_{t+1}, \overline x_t, \overline x_s)$ and say that the common overlapping variable is $x_{t+1} = \pm w_{t-1}$.  If $x_{t+1} = w_{t-1}$ then we need $\overline x_s = \overline w_{t-2}$ to complete the cycle, but $w_{t-2}$ is not part of the hinge from $B$.  If $w_{t+1} = \overline w_{t-1}$, then we need $\overline x_s = w_t$ to complete one cycle and $w_1 = \overline x_t$ to complete the other, but now they are exactly the same 3-cycle - all 3 edges are shared.  The other cases are similar.  
 
 For $j=3$, we note that at least two cycles must come from the same snake, so there are $4$ distinct overlapping variables between them.  And between one snake from $A$ and one from $B$, as in the example above, if one overlapping variable is shared there must be another variable that overlaps outside of the set of $\{ w_1, w_{t-1}, w_t, w_{t+1}, w_s \}$ (or the respective $x$ variables).  Thus the three cycles must in fact have $5$ distinct overlapping variables.  
  
  Finally if  $j=4 $, we must have at least $6$ overlapping variables.  Two cycles from the same snake give 4 distinct variables; each cycle from the other snake introduces an additional variable not in either other snake.  

   All together, this gives that when $r=j$, $p_{0,j} \le ( s^2/(n-s))^{j+1} \le 2^{10} s^{10}  \left(\frac{1}{2n-2s} \right )^{j+1}$ as needed.

Now we consider $i \ge 1$.   We will say a clause $(\overline x_i, x_{i+1})$ from $A$ overlaps \textit{positively} with $B$ if $x_i = w_j$ and $x_{i+1} = w_{j+1}$ for some $j$.  We say the clause overlaps \textit{negatively} if $x_i = \overline w_{j+1}$ and $x_{i+1} = \overline w_j$ for some $j$.  An overlapping \textit{run} of length $r$ will be a maximal sequence of $r$ consecutive overlapping clauses so that back-to-back clauses share a variable; e.g. $(\overline x_i, x_{i+1}), (\overline x_{i+1}, x_{i+2}), \dots, (\overline x_{i+r}, x_{i+r+1})$ that each overlap.  Note that an overlapping run cannot have both negative and positive overlaps, so we can assign each run an orientation.  There are additional possible runs that include $x_t$ or $\overline x_t$, e.g. $(x_t,x_1), (\overline x_1, x_2), \dots$ or $\dots, (\overline x_{s-1}, x_s), (\overline x_s, \overline x_t), (x_t, x_1), \dots$ and at most one of these special runs can branch at $x_t$ and form an `X' or `Y' shape (only if $x_t =w_t$ or $x_t = \overline w_t$).  The following facts can be easily checked:
\begin{itemize}
\item Two distinct runs consist of distinct variables.   

\item Any run of length $l \le t-1$ must involve $l+1$ variables.  

\item $A$ and $B$ may overlap in a single run of length $l \ge t$ that involves $l$ or $l-1$ variables.  For such a run we must have $x_t = w_t$ or $x_t =\overline w_t$.  
\end{itemize}

Now we can compute the probability that $A$ and $B$ overlap in $i$ clauses divided into $k$ runs.  For $i\le t-1$, we have
\begin{align*}
\Pr[ i \text{ clauses in } k \text{ runs} ] \le (s^3)^k 2^4 2^k  \left (\frac{1}{2n -2s} \right) ^{k+i}  
\end{align*}
where the first factor bounds the length and position in $A$ and $B$ of each run, the next factor counts the number of possible ways to branch at $x_t$ or $w_t$ (or if there is an `X' or `Y' shaped run, the ways to choose the active branches), the next factor counts the number of ways to assign an orientation to each run, and the final factor accounts for the probability of the chosen $w_j$ to match the chosen $x_i$.  For $i \ge t$ we must account for the possible special run, and so we have
\begin{align*}
\Pr[ i \text{ clauses in } k \text{ runs} ] \le (s^3)^k 2^4 2^k  \left (\frac{1}{2n -2s} \right) ^{k+i-2}  
\end{align*}

Now we account for the $j$ cycles.   We start with the case $k=1$ and $i\le t-1$.   This single run generates 2 additional overlapping literals (or 3 or 4 if the run is Y or  X-shaped, but in this case the cycles cannot be hinge cycles - the clauses in the hinges overlap already). We start with $j=1$.  If the run is $X$ shaped, the cycle must be of length at least $4$ with $4$ overlapping literals - but these cannot be the same set of $4$ left over form the run: since $i\le t-1$, the indices of the literals from the ends of the run cannot differ by $1$, but in a cycle they do.  So there is at least one additional overlapping variable, and we get a factor of $(s^2/(n-s))$ in the probability.  If the run is not $X$ or $Y$ shaped, we check two cases: if the cycle has four overlapping literals, we get a factor of $(s^2/(n-s))^2 \le s^4 2^2 (1 / (2n -2s))^2$ from the two that cannot be accounted for from the run.  If the cycle is a 3-cycle from a hinge, then we note that at least one of the literals is not from an end of the run - the overlapping literals in a hinge cycle are separated by $t-1$ in both the index from $A$ and the index from $B$, but since $i\le t-1$, the overlapping literals from the run are closer together in both indices.  So again we get a factor of $(s^2/(n-s))$.  

For $j=2$, if either cycle is a 4-cycle, then as above we get a factor of $(s^2/(n-s))^2$.  If both are 3-cycles, they have at least 3 overlapping variables between them.  Each pair of these variables differs by at least $t-2$ in one of the indices.  Since $i \le t-1$, the overlapping end variables from the run cannot include two of these variables, and we get a factor $(s^2/(n-s))^2$.

For $j=3$,  if all cycles are 3-cycles, then there are at least $5$ overlapping variables (see the $j=3, i=0$ analysis) (if one is a 4-cycle we have even more).  At most two may come from the run.  This gives a factor $(s^2/(n-s))^{3} \le 2^{3} s^{6} (1/ (2n -2s))^j$.  

For $j=4$, if all cycles are 3-cycles, then there are at least $6$ overlapping variables.  At most two may come from the run.  This gives a factor $(s^2/(n-s))^{4} \le 2^{4} s^{8} (1/ (2n -2s))^j$.  

Now for $j >4$, we can simply count overlapping literals.  There are at least $4j -8$, and two may come from the run, for a total of $4j -10$, or at least $2j-5$ additional overlapping variables.  This gives a factor $(s^2/(n-s))^{2j-5} \le 2^{5} s^{10} (1/ (2n -2s))^j$.  Overall for $k=1$ we have
\begin{align*}
\Pr[ i \text{ clauses in one run with } j \text{ cycles} ] \le s^{13} 2^{10}  \left (\frac{1}{2n -2s} \right) ^{i+j+1}   
\end{align*}

Now for $k=2$, if $j=1$, we just use the bound above, $s^6 2^6  \left (\frac{1}{2n -2s} \right) ^{i+j+1}$, without an additional factor.  If there is at least one 4-cycle, and at least two cycles, then we have at least $j-1$ additional overlapping variables for $j \le 5$ and at least $2j - 5$ additional for $j >5$, which gives a factor $(s^2/(n-s))^{j-1} $ and $ (s^2/(n-s))^{2j-5}$ respectively, which are both $\le 2^4   s^8 (1/ (2n-2s))^{j-1}$.

For $k=2 , j\ge2$, with two hinge cycles, note that only one variable at the end of a single run of length $\le t-1$ can appear in a hinge cycle.   This means we have at least one additional overlapping variable when $j=2$, three extra overlapping when $j=3$, and four extra overlapping when $j=4$, and with $l$ overlapping variables we gain a factor of $(s^2/(n-s))^l$.  For $j=2, 3,4$ the factor is bounded by $\ 2^4 s^8 (1 / (2n -2s))^{j-1}$.  

All together for $k=2$, we have
\begin{align*}
\Pr[ i \text{ clauses in two runs with } j \text{ cycles} ] \le s^{14} 2^{10}  \left (\frac{1}{2n -2s} \right) ^{i+j+1}   
\end{align*}

$k \ge 3$ proceeds similarly, but now we need only a factor of $(1/ (2n-2s))^{j-2}$.  

  Together we have for $i \le t-1$,

\begin{align*}
\Pr[ i \text{ clauses in } k \text{ runs with } j \text{ cycles} ] \le s^{14}2^{10}  \left (\frac{1}{2n -2s} \right) ^{i+j+1}  
\end{align*}

For $i \ge t$, we only need a rough bound.  Each hinge cycle has at least $2$ overlapping literals and all other cycles have at least $4$.  Each of the $k$ runs can contribute up to $2$ of these literals.  So the total number of additional overlapping literals needed is at least $4j -8 - 2k$ (or $0$ if that is negative), and the total number of additional overlapping variables is at least $2j -4 -k$ (when $k \le 2j- 4$).  This gives a factor of $( s^2/(n-s))^{2j-4-k}$, and so when $k \le 2j-4$,
\begin{align*}
\Pr[ i \text{ clauses in } k \text{ runs with } j \text { cycles} ] &\le s^{4+3k}  2^{6+k}  (2s^2)^{2j - 4k}  \left (\frac{1}{2n -2s} \right) ^{i+2j-6}   \\
&\le  \left (\frac{1}{2n -2s} \right) ^{i+j-6} 
\end{align*}
where the constants in front have been absorbed by the factor $(1/(2n-2s))^j$.  When $k \ge 2j-4$, we have from above
\begin{align*}
\Pr[ i \text{ clauses in } k \text{ runs} ] &\le (s^3)^k 2^4 2^k  \left (\frac{1}{2n -2s} \right) ^{k+i-2}  \\
&\le  2^5 s^3   \left (\frac{1}{2n -2s} \right) ^{i+j-6}
\end{align*}

Now we sum up over all choices of $k$, from $1$ to $s$. For $i \le t-1$:
\begin{align*}
\Pr[ i \text{ clauses with } j \text{ cycles} ] \le s^{15}2^{10}  \left (\frac{1}{2n -2s} \right) ^{i+j+1}  
\end{align*}
and for $i \ge t$, we can bound the sum by
\begin{align*}
\Pr[ i \text{ clauses  with } j \text{ cycles}] \le 2^5 s^4   \left (\frac{1}{2n -2s} \right) ^{i+j -6}  \,.
\end{align*}
\end{proof}

To complete the proof of Theorem \ref{2satthm} we bound $\E (X_s^2)$ in terms of $(\E X_s)^2$:
\begin{align*}
\E (X_s^2) &= \sum_{\text{snakes } A, B} \Pr[A, B \in F_k(n,\gam)] \,
\end{align*}
and
\begin{align*}
(\E X_s)^2 &= \sum_{\text{snakes } A, B} \Pr[A \in F_k(n,\gam)] \Pr[B \in F_k(n,\gam)] =  \sum_{\text{snakes } A, B} \left( \frac{ (2 \gam)^d}{n} \right)^{2(s+1)} \,.
\end{align*}
And so,
\begin{align}
\label{pijeq}
\nonumber
\E (X_s^2) &\le (\E X_s)^2 \sum_{i \ge 0, j \ge 0} p_{ij} \left( \frac{n}{(2 \gam)^d} \right)^{i+j}   \\
\nonumber
&\le  (\E X_s)^2  \left( 1+ \sum_{i=0, j\ge 1} 2^{10} s^{10}  \left(\frac{1}{2n-2s} \right )^{j+1}  \left( \frac{n}{(2 \gam)^d} \right)^{j }  \right .\\
\nonumber
& \left. +  \sum_{ i \le t-1 \atop i+j \ge 1}  2^{10} s^{15}  \left (\frac{1}{2n -2s} \right) ^{i+j+1}   \left( \frac{n}{(2 \gam)^d} \right)^{i+j} \right. \\
\nonumber
& \left. + \sum_{i \ge t, j} 2^{5} s^{4}  \left (\frac{1}{2n -2s} \right) ^{i+j-6}  
 \left( \frac{n}{(2 \gam)^d} \right)^{i+j}   \right)  \\
 \nonumber
&=  (\E X_s)^2 (1+o(1)) \text{ for } \gam> 2^{-(1+1/d)}+ \eps \,.
\end{align}
where we have repeatedly used the fact that $(2 \gam)^d > 2+\eps$.  This proves Proposition \ref{varprop}. An application of Chebyshev's inequality then proves that $X_s \ge 1$ whp.

\subsection*{Proof of Theorem \ref{2satthm} for $F_k(n,\mu)$}

\textit{Lower Bound:} The proof is similar for the $F_k(n, \mu)$ model, with the only difference being accounting for multiple occurrences of the same literal.  The probability of a given $2$-clause $(l_1,l_2)$ being present is $\sim 2^d \mu^2 /n$.  If $X_L$ is the number of directed paths in the implication graph of length $L$, then 
\[ \E X_L \le (2n)^L \left( \frac{2^d \mu^2}{n} \right) ^{L-1} = 2n \left ( 2^{d+1} \mu^2 \right)^{L-1} \]
and for $\mu < 2^{-(d+1)/2} - \eps$ and $L > K \log n$ for large enough $K$, $\E X_L = o(1)$.  

Now we show that whp there is no bicycle of length $\le K \log n$ when $\mu < 2^{-(d+1)/2} - \eps$.  Let $Y_L$ be the number of bicycles of length $L$.  Then
\begin{equation*}
\label{EYLeq}
\E Y_L \le    n^L 2^L (2L)^2 \Pr \left [ (\overline u, w_1),  (w_L, v) \in F  \wedge \bigwedge_{i=1}^{L-1} (\overline w_i, w_{i+1}) \in F  \right ] 
\end{equation*}
where as above $u, v \in \{w_1, \dots, w_L\}$ or their negations.  From Claim \ref{muClaim} above, we have
\[ \E Y_L \le    n^L 2^L (2L)^2    \frac {\mu^2 + 3 \mu +1 } {\mu^2} \left( \frac{2^d \mu^2 }{n} \right) ^{L+1} \,.\]
For $\mu <2^{-(d+1)/2}$, 
\begin{equation*}
\label{MuSmallEq}
\sum_{L=1}^{K \log n} \E Y_L  \le  \sum_{L=1}^{K \log n} \frac {\mu^2 + 3 \mu +1 } {\mu^2} \frac{(2L)^2}{2n}  = O \left( \frac{ \log^3 n}{n} \right) = o(1) \,.
\end{equation*}
Thus whp there are no bicycles in the implication graph for $\mu < 2^{-(d+1)/2} - \eps$, and so $F_k(n, \mu) \in SAT $ whp.

\textit{Upper Bound:}  Again we compute the expectation and variance of $X_s$, the number of $s$-snakes, for $s = \log^ 2 n$.  Using (\ref{wedgeprob}), we have
\begin{align*}
\E X_s &=  \binom n s  s! 2^s \Pr[ (l_1, l_2) \in F_k(n, \mu)]^{s-3} \cdot \Pr[ (l_1,l_2), (l_1, l_3) \in F_k(n,\mu)]^2 \\
& \sim (2n)^s \left( \frac{2^d \mu^2}{n} \right) ^{s-3}  \left(\frac{2^{2d} \mu^2 ( \mu + \mu^2)}{n^2}   \right) ^2 \\
&= \frac{2^{2d+1} (\mu +\mu^2)^2}   {n}  \left( 2^{d+1} \mu ^2  \right) ^{s-1} 
\end{align*}
which tends to $\infty$ as $n \to \infty$ for $s= \log^2 n$ and $\mu > 2^{-(d+1)/2} + \eps$.

To bound the variance of $X_s$, we proceed as above, but we need to account for the fact that clauses in a tree are not independent in this model.   

\begin{claim}
Let $T$ be a set of clauses that form a tree in which at most 2 literals have degree 4, at most 4 have degree 3, and the rest have degree 2 or 1.  If $T$ has $q = O(\log^2 n)$ clauses then the probability that all are present in $F_k(n, \mu)$ is bounded above by:
\[ (\mu^4 + 6 \mu^3 +7 \mu^2 + \mu )^2 (\mu^4 + 3 \mu^3 + \mu)^4 (\mu + 1)^q \left( \frac{2^d \mu^2}{n}  \right)^q (1+ o(1))  \]
\end{claim} 

\begin{proof}
$(\mu^4 + 6 \mu^3 +7 \mu^2 + \mu)$ is the 4th moment of  Poisson random variable with mean $\mu$ and $\mu^4 + 3 \mu^3 +\mu$ is the 3rd moment.  

Let $v_1, \dots v_{q+1}$ be the number of appearances the literals $l_1, \dots l_{q+1}$ of $T$ in the cube.  Given the number of appearances of each literals, the presence of the clauses in a tree are independent events, so the conditional probability that all clauses in $T$ are present is 
\begin{equation*}
\sim \prod_{(l_i, l_j) \in T} \frac{2^d \mu v_i v_j}{n} =\left ( \frac{2^d \mu}{n} \right) ^{q} \prod_{i=1}^{q+1} v_i^{d_i} \,.
\end{equation*}
where $d_i$ is the degree of $l_i$ in $T$. Then taking the expectation over the independent Poisson processes for the $v_i$'s, we have that the probability is
\[ \sim   \left ( \frac{2^d \mu}{n} \right) ^{q} \prod_{i=1}^{q+1} M_{d_i} \]
where $M_r$ is the $r$th moment of a Poisson random variable of mean $\mu$.  The claim now follows from the degree restrictions of $T$.
\end{proof}

For us, the key point of the claim above is that the probability of $q$ edges that form a tree appearing in $F_k(n,\mu)$ is bounded by the product of the probabilities that each appear times $\alpha(\mu)^q$ where $\alpha $ is a constant that depends on $\mu$ but is independent of $n$.  We now modify Claim \ref{overlapclaim1} to count the number of literals that overlap between snakes $A$ and $B$.  

Let $p_{ijl}$ be the probability that $A$ and $B$ overlap on $i$ clauses, form $j$ cycles, and share $l$ literals in addition to those in the overlapping clauses. To prove Proposition \ref{varprop} for the $F_k(n, \mu) $ model, we need to show that, for any $i+j+l \ge 1$, 
\begin{equation}
\label{pijleq}
p_{ijl} \left( \frac{n}{2^d \mu^2} \right)^{i+j} \alpha(\mu)^l = o(s^{-3}) \,.
\end{equation}
where $2 s^3$ is an upper bound on the number of possible values for $i,j, l$.
  The arguments in Claim \ref{overlapclaim1} suffice in this model as well, for any overlapping literal that shares an underlying variable with an overlapping clause or an overlapping literal in a cycle: the bounds in the proof of Claim \ref{overlapclaim1} are strong enough to dominate another constant factor multiple.  All that remain are overlapping literals whose variables appear in neither overlapping clauses nor in cycles - these must come in pairs, and using the fact that the probability of $r$ variables overlapping is $\le (s^2/(n-s))^r$, we have factors that are all $O(\log^4 n \cdot n^{-1/2})$, which is enough for (\ref{pijleq}).

\section{Statement and Proof of Proposition \ref{ksatthm}}
\label{sec:firstmomSec}

For $k \ge 3$ we give bounds on the satisfiability threshold, showing in particular that the transition from almost certain satisfiability to almost certain unsatisfiability occurs when the number of clauses is linear in the number of variables:

\begin{prop}
\label{ksatthm}
For all $k \ge 3$ there exist functions $\overline \gam(k), \underline \gam(k), \overline \mu(k), \underline \mu(k)$ so that for any $\eps>0$,
\begin{enumerate}
\item For $\gam < \underline \gam(k) -\eps $, whp $F_k(n,\gam) \in SAT$. For $\gam> \overline \gam(k) +\eps$, whp $F_k(n,\gam) \notin SAT$.
\item For $\mu < \underline \mu(k) -\eps $, whp $F_k(n,\mu) \in SAT$. For $\mu > \overline \mu(k) +\eps$, whp $F_k(n,\mu) \notin SAT$.
\end{enumerate}
We can take $\underline \gam(k) =  2^{-(1+1/d)}$, $\underline \mu(k) = 2^{-(d+1)/2} $, $\overline \gam(k) = ( k-1)^{1/d}$, and $\overline \mu(k) = k + \ln 2$. In particular, all functions are independent of $n$ and so the threshold for satisfiability occurs with a linear number of clauses.
\end{prop}

\paragraph{$F_k(n,\gam)$}
The lower bound follows from the lower bound in Theorem \ref{2satthm}.  For the same set of points in the cube, form both the corresponding $2$-SAT formula and the $k$-SAT formula.  For each $k$-clause the $2$-SAT formula will include each of the $\binom k 2$ subclauses of length $2$.  If there is a satisfying assignment to the $2$-SAT formula, the same assignment will satisfy the $k$-SAT formula.    

For an upper bound, we use the first-moment method and show that the expected number of satisfying assignments is $o(1)$.  This will follow if we show that the probability that the all T assignment is satisfying is $\le q^n$ for some $q < 1/2$ independent of $n$.  The all T assignment is satisfying if and only if no $k$-clause of all negative literals is present, so we need an upper bound on the probability that  $n$ points uniformly distributed on $[0,1]^d$ has no set of $k$ points in a ball of radius $\gamma n^{-1/d}$.

Set $\gam > (k-1)^{1/d}+\eps$. Tile $[0,1]^d$ by $(\lceil n^{1/d}/ \gam \rceil )^d$ boxes of side length $\gam n^{1/d}$ (with boxes along the boundary possibly smaller).  For large enough $n$ (depending on $\eps$), the number of boxes is strictly less than $ n /(k-1)$.  By the pigeonhole principle there must be a box with at least $k$ points, and so the probability of no $k$-cliques is $0$.    This is true for any set of $n$ literals, and so with probability $1$ there is no satisfying assignment.

\paragraph{$F_k(n,\mu)$}

The lower bound again follows from the $k=2$ case and Theorem \ref{2satthm}.  For the upper bound, tile $[0,1]^d$ by $n$ boxes of side length $ n^{-1/d}$.  The probability that there is no $k$-clause of negative literals is bounded by the probability that none of these boxes contain $k$ negative literals.  The nodes in the different boxes are independent, so we need to show that for large enough $\mu$, the probability there are fewer than $k$ negative literals in a single cube of side length $n^{-1/d}$ is strictly less than $1/2$.  The number of negative literals in a single such cube has distribution $\poiss(\mu)$. The median of a Poisson with mean $\lam$ is at least $\lam - \ln 2$, so if we pick $\overline \mu(k) > k + \ln 2$, then $\Pr[ \poiss(\mu) < k] < 1/2$ and via a first-moment argument whp $F_k(n, \mu)$ is unsatisfiable.

\section*{Acknowledgements}
The authors would like to thank Alfredo Hubard for many interesting conversations on this topic.

\bibliographystyle{plain}	
\bibliography{geomSAT}

\end{document}